\documentclass[12pt]{article}
\usepackage{fancyhdr,fancyref,amsmath,amssymb,latexsym,float,graphicx,tikz,eufrak,comment}
\newtheorem{theorem}{Theorem}[section]
\newtheorem{lemma}[theorem]{Lemma}
\newtheorem{proposition}[theorem]{Proposition}
\newtheorem{corollary}[theorem]{Corollary}
\newtheorem{conjecture}[theorem]{Conjecture}

\newenvironment{proof}[1][Proof]{\begin{trivlist}
\item[\hskip \labelsep {\bfseries #1}]}{\end{trivlist}}
\newenvironment{definition}[1][Definition]{\begin{trivlist}
\item[\hskip \labelsep {\bfseries #1}]}{\end{trivlist}}

\newenvironment{remark}[1][Remark]{\begin{trivlist}
\item[\hskip \labelsep {\bfseries #1}]}{\end{trivlist}}

\newcommand{\qed}{\nobreak \ifvmode \relax \else
      \ifdim\lastskip<1.5em \hskip-\lastskip
      \hskip1.5em plus0em minus0.5em \fi \nobreak
      \vrule height0.75em width0.5em depth0.25em\fi}
\begin{document}

\title{Adjustment matrices for the principal block of the Iwahori-Hecke algebra $\mathcal{H}_{5e}$}
\author{Aaron Yi Rui Low}
\date{June 12, 2020}
\maketitle
\textit{Department of Mathematics, National University of Singapore,}

\textit{Block S17, 10 Lower Kent Ridge Road, 119076 Singapore}\par\nopagebreak
E-mail: \textit{aaronlyr94@gmail.com}

\begin{abstract}
James's Conjecture predicts that the adjustment matrix for blocks of the Iwahori-Hecke algebra of the symmetric group is the identity matrix when the weight of the block is strictly less than the characteristic of the field. In this paper, we consider the case when the characteristic of the field is greater than or equal to 5, and prove that the adjustment matrix for the principal block of $\mathcal{H}_{5e}$ is the identity matrix whenever $e\neq4$. When $e=4$, we are able to calculate all but two entries of the adjustment matrix.
\end{abstract}

\section{Introduction}
Suppose $q$ is a non-zero element of a field $\mathbb{F}$. The \textit{Iwahori-Hecke algebra} $\mathcal{H}_{\mathbb{F},q}(\mathfrak{S}_{n})$ of the symmetric group $\mathfrak{S}_{n}$, over $\mathbb{F}$ and with parameter $q$ is the unital associative $\mathbb{F}$-algebra with generators $T_{1},T_{2},\dots,T_{n-1}$ subject to the following relations:
\begin{align*} 
\bullet       &(T_{i}-q)(T_{i}+1)=0,                                 &i\in\{1,2,\dots,n-1\}.\\
\bullet         &T_{i}T_{j}=T_{j}T_{i},                          &\mid i-j\mid>1. \\
\bullet          &T_{i}T_{i+1}T_{i}=T_{i+1}T_{i}T_{i+1},    &i\in\{1,2,\dots,n-2\}.
\end{align*} 
When there is no ambiguity, we denote $\mathcal{H}_{\mathbb{F},q}(\mathfrak{S}_{n})$ by $\mathcal{H}_{n}$ . Let $e$ be the smallest integer such that $1+q+q^{2}+\dots+q^{e-1}=0$, assuming throughout the paper that it exists. If $q=1$, $\mathcal{H}_{n}\simeq\mathbb{F}\mathfrak{S}_{n}$ and $e$ is just the characteristic of $\mathbb{F}$. To each partition $\lambda$ of $n$, we associate a \textit{Specht Module} $S^{\lambda}$ for $\mathcal{H}_{n}$. A partition is \textit{e-singular} if it has $e$ parts of the same size. It is called \textit{e-regular} otherwise. For an $e$-regular partition $\lambda$, $S^{\lambda}$ has an irreducible cosocle $D^{\lambda}$. The set of $D^{\lambda}$ as $\lambda$ ranges over all $e$-regular partitions gives a complete set of distinct irreducible $\mathcal{H}_{n}$-modules. We denote the projective cover of $D^{\lambda}$ by $P^{\lambda}$. The composition factors $[S^{\lambda}:D^{\mu}]=[P^{\mu}:S^{\lambda}]$ are called the \textit{decomposition numbers} of $\mathcal{H}_{n}$. Typically, they are recorded in a \textit{decomposition matrix} with rows indexed by partitions of $n$ and columns indexed by $e$-regular partitions of $n$, whose $(\lambda,\mu)$-entry is $
[S^{\lambda}:D^{\mu}]$. 

One of the most important outstanding problems in the modular representation theory of the symmetric groups is to determine the decomposition numbers. When the field is $\mathbb{C}$, there is an algorithm for calculating the decomposition numbers for the Iwahori-Hecke algebras. It is known that the decomposition matrix for fields of prime characteristic may be obtained from that of $\mathbb{C}$ by post-multiplying by an `adjustment matrix'. Therefore, we often work with adjustment matrices instead of the decomposition matrices directly when the characteristic of the field is prime. Other than for some cases with small weight, there is a great deal not known about the adjustment matrices. James's Conjecture predicts that the adjustment matrix for a block of $\mathcal{H}_{n}$ is the identity matrix when the characteristic of the field is strictly less than the weight of that block. The conjecture has been proven for weights up to four by the works of Richards~\cite{Richards} and Fayers~\cite{wt 4,wt 3}. However, Williamson found a counter-example~\cite{Williamson} to James's Conjecture. Nevertheless, the smallest counter-example produced in his paper occurs in the symmetric group $\mathfrak{S}_{n}$ where $n=1744860$. There is considerable interest in finding smaller counter-examples. 

In section 3, we prove that the adjustment matrix for the principal block of $\mathcal{H}_{5e}$ is the identity matrix when $\textnormal{char}(\mathbb{F})\ge5$ and $e\neq4$. When $\textnormal{char}(\mathbb{F})\ge5$ and $e=4$, we show that all but 2 off-diagonal entries of the adjustment matrix for the principal block of $\mathcal{H}_{20}$ are zero; these 2 entries are not yet known and may be explored in the future. Nevertheless, if one is interested in specialising to the symmetric groups, it is sufficient to consider $e\ge5$. It is hoped that some of the techniques used in this paper could be generalised to higher weights. In section 3, the case of $\mathcal{H}_{25}$ when $\textnormal{char}(\mathbb{F})=5$ is perhaps the most interesting. In this case, the defect group of the principal block of $\mathcal{H}_{25}=\mathbb{F}\mathfrak{S}_{25}$ is not Abelian, and experts expect it to behave differently from $\mathcal{H}_{25}$ in characteristic zero. On the other hand, Fayers's extension of James's Conjecture~\cite[Conjecture 3.1]{extended James} suggests that the decomposition numbers of these two blocks are the same. In the next section, we lay the groundwork that we need for section 3. 
\section{Background and techniques}
\subsection{Blocks of $\mathcal{H}_{n}$ and abacus displays}
Take an abacus with $e$ vertical runners, numbered $0,\dots,e-1$ from left to right, marking positions $0,1,\dots$ on the runners increasing from left to right along successive `rows'. Given a partition $\lambda$ of $n$, take an integer $r\ge\lambda_{1}'$, the number of parts of $\lambda$. Define $\beta_{i}=\lambda_{i}+r-i$ for $i\in\{1,\dots,r\}$. Now, place a bead at position $\beta_{i}$ for each $i$. The resulting configuration is called the abacus display for $\lambda$. We remark that moving a bead up one place on its runner is akin to removing an $e$-hook from the young diagram of $\lambda$. By moving all the beads as high as possible on their runners, the resulting configuration is the abacus display for the \textit{e-core} of $\lambda$. 

\begin{theorem} (Nakayama Conjecture, ~\cite[Corollary 5.38]{mathas book})
Let $\lambda$ and $\mu$ be partitions of $n$. Then, $S^{\lambda}$ and $S^{\mu}$ lie in the same block of $\mathcal{H}_{n}$ if and only if $\lambda$ and $\mu$ have the same $e$-core.
\end{theorem}
Therefore, we may define the $e$-weight and $e$-core of a block of $\mathcal{H}_{n}$ simply to be the $e$-weight and $e$-core of a partition lying in that block. 
Let $\lambda(i)$ be the partition corresponding to the abacus display containing only a single runner, the $i^{th}$ runner. Denote the number of beads in the $i^{th}$ runner as $b_{i}$. Then, we may write $\lambda$ as $$\langle 0_{\lambda(0)},\dots,(e-1)_{\lambda(e-1)} \mid b_{0},\dots,b_{e-1}\rangle;$$
we omit $i_{\lambda(i)}$ if $\lambda(i)=\varnothing$ and omit $\lambda(i)$ if $\lambda(i)=(1)$. Additionally, we may omit $b_{0},\dots,b_{e-1}$ if it is clear which block we are dealing with. If $\lambda$ lies in the block $B$ of $\mathcal{H}_{n}$, we say that $B$ is the block of $e$-weight $w$ with the $\langle b_{0},\dots,b_{e-1}\rangle$ notation.

\subsection{Modular Branching Rules}
We use some notational conventions for modules. We write
$$M\sim M_{1}^{a_{1}}+M_{2}^{a_{2}}+\dots+M_{r}^{a_{r}}$$
to indicate that $M$ has a filtration in which the factors are $M_{1},\dots,M_{r}$ appearing $a_{1},\dots,a_{r}$ times respectively. Additionally, we write $M^{\oplus a}$ to indicate the direct sum of $a$ isomorphic copies of $M$.

There is a natural embedding $\mathcal{H}_{n-1}\le \mathcal{H}_{n}$. If $M$ is a module for $\mathcal{H}_{n}$, the restriction of M to $\mathcal{H}_{n-r}$ is denoted by $M\downarrow_{\mathcal{H}_{n-r}}$. Similarly, the induction of M to $\mathcal{H}_{n+r}$ is denoted by $M\uparrow_{\mathcal{H}_{n+r}}$. If $B$ is a block of $\mathcal{H}_{n-r}$, we write $M\downarrow_{B}$ to indicate the projection of $M\downarrow_{\mathcal{H}_{n-r}}$ onto $B$. Similarly, if $C$ is a block of $\mathcal{H}_{n+r}$, we write $M\uparrow_{B}$ to indicate the projection of $M\uparrow_{\mathcal{H}_{n+r}}$ onto $C$. In this section, we describe the restriction and induction of Specht modules and simple modules.

Suppose $A$, $B$ and $C$ are blocks of $\mathcal{H}_{n-\kappa}$, $\mathcal{H}_{n}$ and $\mathcal{H}_{n+\kappa}$ respectively, and that there is an integer $i$ such that an abacus display for $A$ is obtained from that of $B$ by moving exactly $\kappa$ beads from runner $i$ to runner $i-1$, while an abacus display for $C$ is obtained from that of $B$ by moving exactly $\kappa$ beads from runner $i-1$ to runner $i$.

Suppose $\lambda$ is a partition in $B$, and that $\lambda^{-1},\lambda^{-2},\dots,\lambda^{-r}$ are the partitions in $A$ that may be obtained from $\lambda$ by moving exactly $\kappa$ beads on runner $i$ one place to the left. Similarly, let $\lambda^{+1},\lambda^{+2},\dots,\lambda^{+r}$ be the partitions in $C$ that may be obtained from $\lambda$ by moving exactly $\kappa$ beads on runner $i-1$ one place to the right. We have the following result.

\begin{theorem}(The Branching Rule ~\cite[Corollary 6.2]{mathas book})
Suppose $A$, $B$, $C$ and $\lambda$ are as above. Then,
$$S^{\lambda}\downarrow^{B}_{A} \sim (S^{\lambda^{-1}})^{\kappa !}+ (S^{\lambda^{-2}})^{\kappa !}+\dots+ (S^{\lambda^{-r}})^{\kappa !}$$ and
$$S^{\lambda}\uparrow^{C}_{B} \sim (S^{\lambda^{+1}})^{\kappa !}+ (S^{\lambda^{+2}})^{\kappa !}+\dots+ (S^{\lambda^{+r}})^{\kappa !}.$$
\end{theorem}

For the discussion of the restriction and induction of simple modules, we assume that $\lambda$ is $e$-regular. The \textit{i-signature} of $\lambda$ is the sequence of signs defined as follows. Starting from the top row of the abacus display for $\lambda$ and working downwards, write a $-$ if there is a bead on runner $i$ but no bead on runner $i-1$; write a $+$ if there is a bead on runner $i-1$ but no bead on runner $i$; write nothing for that row otherwise. Given the $i$-signature of $\lambda$, successively delete all neighbouring pairs of the form $-+$ to obtain the \textit{reduced i-signature} of $\lambda$. If there are any $-$ signs in the reduced $i$-signature of $\lambda$, we call the corresponding beads on runner $i$ \textit{normal}; if there are at least $\kappa$ normal beads, then we define $\lambda^{-}$ to be the partition obtained by moving the $\kappa$ highest normal beads one place to the left. If there are any $+$ signs in the reduced $i$-signature, we call the corresponding beads on runner $i-1$ \textit{conormal}; if there are at least $\kappa$ conormal beads, then we define $\lambda^{+}$ to be the partition obtained by moving the $\kappa$ lowest conormal beads one place to the right.
\begin{theorem}(~\cite[\S 2.5]{branching})
Suppose $A$,$B$ and $\lambda$ are as above.
\begin{itemize}
\item If there are fewer than $\kappa$ normal beads on runner $i$ of the abacus display for $\lambda$, then $D^{\lambda}\downarrow^{B}_{A}=0$.
\item If there are exactly $\kappa$ normal beads on runner $i$ of the abacus display for $\lambda$, then $D^{\lambda}\downarrow^{B}_{A}\cong(D^{\lambda^{-}})^{\oplus\kappa !}$.
\item If there are fewer than $\kappa$ conormal beads on runner $i-1$ of the abacus display for $\lambda$, then $D^{\lambda}\uparrow^{C}_{B}=0$.
\item If there are exactly $\kappa$ conormal beads on runner $i-1$ of the abacus display for $\lambda$, then $D^{\lambda}\uparrow^{C}_{B}\cong(D^{\lambda^{+}})^{\oplus\kappa !}$.
\end{itemize}
\end{theorem}

\subsection{$v$-decomposition numbers}
Let $\mathcal{P}$ be the set of all partitions. Let the quantum affine algebra, $U_{v}(\hat{\mathfrak{sl}}_{e})$ be the associative algebra over $\mathbb{C}(v)$ with generators $e_{i},f_{i},k_{i},k^{-1}_{i} (0\le i \le e-1),d,d^{-1}$ subject to some relations (see ~\cite[\S4]{LLT}). The Fock space representation $\mathcal{F}$ is the $U_{v}(\hat{\mathfrak{sl}}_{e})$-module with basis $\{s(\mu) : \mu\in\mathcal{P}\}$ as a $\mathbb{C}(v)$-vector space. Let $L$ be the free $\mathbb{Z}[v]$-lattice in $\mathcal{F}$ generated by $\{s(\nu) : \nu\in\mathcal{P}\}$. Moreover, let $x\mapsto \overline{x}$ be the bar involution on $\mathcal{F}$ (see ~\cite[\S6]{LLT}) having the following (among other) properties:
\begin{itemize}
\item$\overline{b(v)x}=b(v^{-1})\overline{x}$   \hspace{1cm}   $\forall b(v)\in\mathbb{C}(v)$,\hspace{7mm}   $\forall   x\in\mathcal{F}$.
\item$\overline{f_{i}(x)}=f_{i}(\overline{x})$   \hspace{1cm}   $\forall x\in\mathcal{F}$.
\end{itemize}
$\mathcal{F}$ has a distinguished basis $\{G(\mu)\mid \mu\in\mathcal{P}\}$, called the canonical
basis satisfying:
\begin{itemize}
\item $\overline{G(\mu)}=G(\mu)$
\item $G(\mu)\equiv s(\mu)$ mod $vL$.
\end{itemize}
The \textit{v-decomposition number} $d^{(e)}_{\lambda\mu}(v)$ is the coefficient of $s(\lambda)$ in $G(\mu)$. 
Lascoux, Leclerc and Thibon have come up with the LLT algorithm ~\cite{llt}, a recursive algorithm for computing the canonical basis. The following theorem due to Ariki explains the connection between the $v$-decomposition numbers and the decomposition numbers.

\begin{theorem} ~\cite[Theorem 4.4]{ariki}
Let $\lambda$ and $\mu$ be partitions of $n$, with $\mu$ $e$-regular. Then,
$$[S^{\lambda}_{\mathbb{C},\zeta}:D^{\mu}_{\mathbb{C},\zeta}]=d^{(e)}_{\lambda\mu}(1).$$
\end{theorem}
Consequently, the decomposition matrix for $\mathcal{H}_{\mathbb{C},\zeta}(\mathfrak{S}_{n})$ can be computed by the LLT algorithm.

Fix any field $\mathbb{F}$. Let $\mathcal{G}^{p}_{n}(\mathbb{F})$ be the \textit{Grothendieck group} (see ~\cite[Chapter 6, \S1.1]{mathas book}) of finitely generated projective $\mathcal{H}_{\mathbb{F},q}(\mathfrak{S}_{n})$-modules with complex coefficients; that is the additive abelian group (with complex coefficients) generated by the symbols $[\![P]\!]^{p}$, where $P$ runs over the isomorphism classes of finitely generated projective $\mathcal{H}_{n}$-modules. These elements satisfy the relations $[\![P]\!]^{p}=[\![M]\!]^{p}+[\![N]\!]^{p}$ whenever $P=M\oplus N$. Therefore, the set of $[\![P^{\lambda}_{\mathbb{F}}]\!]^{p}$ as $\lambda$ runs over all $e$-regular partitions of $n$ forms a basis of $\mathcal{G}^{p}_{n}(\mathbb{F})$.

Let $\mathcal{E}_{n}$ be the complex vector space with basis the set of symbols $[\![S^{\nu}]\!]$ where $\nu$ runs over all partitions of $n$. ($\mathcal{E}_{n}$ is the Grothendieck group of a semi-simple Iwahori-Hecke algebra.) Recall that $[S^{\lambda}_{\mathbb{F}}:D^{\mu}_{\mathbb{F}}]=[P^{\mu}_{\mathbb{F}}:S^{\lambda}_{\mathbb{F}}]$. There is an injective homomorphism of abelian groups $\textbf{e}_{\mathbb{F}}:\mathcal{G}^{p}_{n}(\mathbb{F})\rightarrow\mathcal{E}_{n}$ determined by
$$\textbf{e}_{\mathbb{F}}[\![P^{\lambda}_{\mathbb{F}}]\!]^{p}=\sum\limits_{\nu\vdash n}[S^{\nu}_{\mathbb{F}}:D^{\lambda}_{\mathbb{F}}][\![S^{\nu}]\!].$$

Suppose that $A$ and $B$ are blocks of $\mathcal{H}_{n}$ and $\mathcal{H}_{n+1}$ respectively, and that an abacus display with $r$ beads for $B$ is obtained from that for $A$ by moving a bead from runner $k-1$ to runner $k$. Let $i$ be the residue of $k-r$ modulo $e$. We define \textit{i}-Ind to be the group homomorphism from $\mathcal{G}^{p}_{n}(\mathbb{F})$ to $\mathcal{G}^{p}_{n+1}(\mathbb{F})$ taking $[\![P]\!]^{p}$ to $[\![P\uparrow_{A}^{B}]\!]^{p}$. By abusing notation, we also refer to \textit{i}-Ind as the group homomorphism from $\mathcal{E}_{n}$ to $\mathcal{E}_{n+1}$ taking $[\![S^{\nu}]\!]$ to $[\![S^{\nu}\uparrow_{A}^{B}]\!]$.

We now describe the action of $f_{i}$ on $s(\lambda)$. Display $\lambda$ on an abacus with $e$ runners and $r$ beads, where $r\ge \lambda_{1}'$. Let $k$ be the residue class of $(i+r)$ modulo $e$. Suppose there is a bead on runner $k-1$ whose succeeding position on runner $k$ is vacant. Let $\mu$ be the partition whose abacus display is obtained by moving such a bead to its succeeding position. Define $N_{i}(\lambda,\mu)$ to be the number of beads on runner $k-1$ below the bead moved to obtain $\mu$ minus the number of beads on runner $k$ below the vacant position that becomes occupied in obtaining $\mu$. Then,
$$f_{i}(s(\lambda))=\sum\limits_{\mu}v^{N_{i}(\lambda,\mu)}s(\mu).$$
Note that when $v=1$, $f_{i}$ acts in the same way as \textit{i}-Ind on $\mathcal{E}_{n}$.
\begin{proposition}(~\cite[Proposition 2.4]{positive})\\
If we write $f_{i}(G(\mu))$ in the form $$f_{i}(G(\mu))=\sum\limits_{\nu}a_{\nu}(v)G(\nu),$$ then $a_{\nu}(v)\in\mathbb{N}_{0}[v+v^{-1}]$ for all $\nu$.
\end{proposition}

\subsection{Adjustment Matrices and James's Conjecture}
Denote $\mathcal{H}_{\mathbb{C},\zeta}(\mathfrak{S}_{n})$ by $\mathcal{H}_{n}^{0}$ and $\mathcal{H}_{\mathbb{F},q}(\mathfrak{S}_{n})$ by $\mathcal{H}_{n}$. The Specht modules corresponding to two partitions lie in the same block of $\mathcal{H}_{n}$ if and only if they lie in the same block of $\mathcal{H}_{n}^{0}$ by Nakayama's lemma. Therefore, given a block $B$ of $\mathcal{H}_{n}$, we may denote $B^{0}$ to be its corresponding block in $\mathcal{H}_{n}^{0}$.

\begin{theorem}~\cite[Theorem 6.35]{mathas book}
Let $D$ and $D^{0}$ be the decomposition matrices for the blocks $B$ and $B^{0}$ respectively. Then, there is a square matrix $A$ with non-negative integer entries such that $$D=D^{0}A.$$
\end{theorem}
We call $A$ the \textit{adjustment matrix} for the block $B$ and as a shorthand, denote its $(\lambda,\mu)$-entry as $\textnormal{adj}_{\lambda\mu}$.
Since $D^{0}$ can be computed by the LLT algorithm, $D$ is often studied by considering its adjustment matrix.
\begin{conjecture}(James's Conjecture ~\cite[\S4]{glnq})
Let $B$ be a block of $\mathcal{H}_{n}$ of $e$-weight $w$. If $w<\textnormal{char}(\mathbb{F})$, then the adjustment matrix for the block $B$ is the identity matrix.
\end{conjecture}
\begin{theorem} (~\cite[Theorem 2.5, Theorem 2.6]{wt 4}) 
Suppose $\textnormal{char}(\mathbb{F})\ge 5$, and that $B$ is a block of $\mathcal{H}_{n}$ of weight at most 4. Then, the adjustment matrix for $B$ is the identity matrix.
\end{theorem}
The conjecture has been proved for weights at most four. In this paper, we prove the conjecture and its extension by Fayers when $e\neq4$ for the principal block of $\mathcal{H}_{5e}$ which has \textit{e-weight} equal to 5.
\begin{theorem}
Suppose $e\neq4$ and $\textnormal{char}(\mathbb{F})$$\ge5$, then the adjustment matrix for the principal block of $\mathcal{H}_{5e}$ is the identity matrix.
\end{theorem}

\subsection{The Mullineux map}
Let $T_{1},\dots,T_{n-1}$ be the standard generators of $\mathcal{H}_{n}$ defined at the beginning of this section. Let $\sharp:\mathcal{H}_{n}\rightarrow\mathcal{H}_{n}$ be the involutory automorphism sending $T_{i}$ to $q-1-T_{i}$. Given a $\mathcal{H}_{n}-module$ $M$, define $M^{\sharp}$ 
to be the module with the same underlying vector space and with action
$$h\cdot m=h^{\sharp}m.$$
In the case of the symmetric groups when $q=1$, $M^{\sharp}$ is $M\otimes \textnormal{sgn}$, where sgn is the 1-dimensional signature representation.
Let $\lambda^{\diamond}$ be the $e$-regular partition such that $(D^{\lambda})^{\sharp}\cong D^{\lambda^{\diamond}}$. The map $\lambda \mapsto \lambda^{\diamond}$ is an involutory bijection from the set of $e$-regular partitions of $n$ to itself, and is given combinatorially by Mullineux's algorithm~\cite{Mullineux} which depends only on $\lambda$ and $e$, not $\mathbb{F}$ and $q$.
\begin{proposition}
(~\cite[Lemma 4.2]{wt 3}) If $\lambda$ and $\mu$ are e-regular partitions of n, then $\textnormal{adj}_{\lambda\mu}=\textnormal{adj}_{\lambda^{\diamond}\mu^{\diamond}}$.
\end{proposition}
\begin{remark}
This almost halves the number of entries of the adjustment matrix that we need to calculate in section 3.
\end{remark}

\subsection{The Jantzen-Schaper formula}
Let $\lambda$ be a partition and consider its abacus display, say with $k$ beads. Suppose that after moving a bead at position $a$ up its runner to a vacant position $a-ie$, we obtain the partition $\mu$. Denote $l_{\lambda\mu}$ for the number of occupied positions between $a$ and $a-ie$, and let $h_{\lambda\mu}=i$. \\
Further, write $\lambda \xrightarrow{\text{$\mu$}} \tau$ if the abacus display of $\tau$ with $k$ beads is obtained from that of $\mu$ by moving a bead at position $b-ie$ to a vacant position $b$, and $a<b$.
\begin{definition} \textit{Jantzen-Schaper bound}\\
Let $p=\textnormal{char}(\mathbb{F})$.
$$J_{\mathbb{F}}(\lambda,\mu)=\sum\limits_{\tau,\sigma}(-1)^{l_{\lambda \sigma}+l_{\tau \sigma} +1}(1+v_{p}(h_{\lambda \sigma}))[S^{\tau}_{\mathbb{F}}:D^{\mu}_{\mathbb{F}}],$$
where the sum runs through all $\tau$ and $\sigma$ such that $\lambda \xrightarrow{\text{$\sigma$}} \tau$, and where $v_{p}$ denotes the standard p-valuation if $p>0$ and $v_{0}(x)=0$  $\forall x$.
\end{definition}

\begin{theorem} Jantzen-Schaper formula(~\cite[Theorem 4.7]{Jantzen Schaper})
$$[S^{\lambda}_{\mathbb{F}}:D^{\mu}_{\mathbb{F}}] \le J_{\mathbb{F}}(\lambda,\mu).$$
Moreover, the left-hand side is zero if and only if the right-hand side is zero.
\end{theorem}

\begin{corollary}
If $J_{\mathbb{F}}(\lambda,\mu) \le 1$, then 
$$[S^{\lambda}_{\mathbb{F}}:D^{\mu}_{\mathbb{F}}] = J_{\mathbb{F}}(\lambda,\mu).$$
\end{corollary}
We write $\lambda \rightarrow \tau$ if there exists some $\mu$ such that $\lambda \xrightarrow{\text{$\mu$}} \tau$. Further, write $\lambda <_{J}\sigma$ if there exist partitions $\tau_{0},\tau_{1},\dots,\tau_{r}$ such that $\tau_{0}=\lambda$, $\tau_{r}=\sigma$ and $\tau_{i-1}\rightarrow \tau_{i}$ $\forall i \in\{1,2,\dots,r\}$. We call $\le_{J}$ the \textit{Jantzen order} and it is clear that this defines a partial order on the set of all partitions, and that only partitions in the same block are comparable under this partial order. Moreover, the usual dominance order extends the Jantzen order. Combined with the fact that $\mathcal{H}_{n}$ is a cellular algebra, we have the following theorem.
\begin{theorem} Suppose $\lambda$ and $\mu$ are partitions of n, with $\mu$ e-regular. Then,
\begin{itemize}
\item $[S^{\mu}:D^{\mu}]=1;$
\item $[S^{\lambda}:D^{\mu}]>0 \Rightarrow \mu \ge_{J} \lambda .$
\end{itemize}
\end{theorem}

\begin{corollary}
Suppose $\lambda$ and $\mu$ are e-regular partitions lying in a block B of $\mathcal{H}_{n}$. Then,
\begin{itemize}
\item $\textnormal{adj}_{\mu \mu} = 1;$
\item $\textnormal{adj}_{\lambda \mu}>0 \Rightarrow \mu  \ge_{J} \lambda .$
\end{itemize}
\end{corollary}

It is difficult to check that $\mu\ngtr_{J}\lambda$ by inspection. To this end, we introduce the \textit{product order} on partitions. Let $\lambda$ be a partition, displayed on an abacus with $e$ runners and $N$ beads. Suppose that the beads having positive \textit{e-weights} are at positions $a_{1},a_{2},\dots,a_{r}$ with weights $w_{1},w_{2},\dots,w_{r}$ respectively. The \textit{induced e-sequence} of $\lambda$, denoted $s(\lambda)_{N}$, is defined as 
$$\bigsqcup\limits_{i=1}^{r}(a_{i},a_{i}-e,\dots,a_{i}-(w_{i}-1)e),$$ where $(b_{1},b_{2},\dots,b_{s})\sqcup(c_{1},c_{2},\dots,c_{t})$ denotes the weakly decreasing sequence obtained by rearranging terms in the sequence $(b_{1},\dots,b_{s},c_{1},\dots,c_{t})$. Note that $s(\lambda)_{N}\in\mathbb{N}_{0}^{w}$, where $w$ is the $e$-weight of $\lambda$. 

We define a partial order $\ge_{P}$ on the set of partitions by: $\mu\ge_{P}\lambda$ if and only if $\mu$ and $\lambda$ have the same $e$-core and $e$-weight, and $s(\mu)_{N}\ge s(\lambda)_{N}$ (for sufficiently large $N$) in the standard product order of $\mathbb{N}_{0}^{w}$.

\begin{lemma}(~\cite[Lemma 2.9]{beyond})
$$\lambda\le_{J}\mu\Rightarrow\lambda\le_{P}\mu.$$
\end{lemma}
Therefore, $\mu\ngtr_{P}\lambda \Rightarrow \textnormal{adj}_{\lambda\mu}=0$.

An important connection between $v$-decomposition numbers and the Jantzen-Schaper formula was shown by Ryom-Hansen in the following theorem:
\begin{theorem}
(~\cite[Theorem 1]{Hansen}) Suppose $\lambda$ and $\mu$ are partitions of n, with $\mu$ e-regular, and let $d_{\lambda \mu}^{(e)'}(v)$ denote the derivative of the v-decomposition number $d_{\lambda \mu}^{(e)}(v)$ with respect to v. Then
$$J_{\mathbb{C}}(\lambda,\mu)=d_{\lambda \mu}^{(e)'}(1).$$
\end{theorem}
\begin{remark}
If we fix a particular value of $e$, it is often easier to get the $v$-decomposition numbers using the LLT algorithm as opposed to finding the Jantzen-Schaper bound directly.
\end{remark}

\begin{corollary}
Suppose $\lambda$ and $\mu$ are $e$-regular partitions lying in the principal block B of $\mathcal{H}_{5e}$, $e\ge5$ and $p=\textnormal{char}(\mathbb{F})\ge5$. Moreover, suppose that $\lambda$ is not of the form $\langle i_{5} \rangle$. Additionally, suppose that $\textnormal{adj}_{\nu\mu}=0$ for all e-regular partitions $\nu$ such that $\lambda <_{J} \nu <_{J} \mu$, and that $d_{\lambda \mu}^{(e)}(v) \in \{0,v\}$. Then, $\textnormal{adj}_{\lambda\mu}=0.$
\end{corollary}
\begin{proof}
Suppose $\lambda <_{J} \nu <_{J} \mu$. Then,
$$[S^{\nu}_{\mathbb{F}}:D^{\mu}_{\mathbb{F}}]=\sum\limits_{\nu\le_{J} \sigma \le_{J} \mu} \textnormal{adj}_{\sigma \mu}[S^{\nu}_{\mathbb{C}}:D^{\sigma}_{\mathbb{C}}]=[S^{\nu}_{\mathbb{C}}:D^{\mu}_{\mathbb{C}}],$$
where the first equality is due to the definition of adjustment matrices, Theorem 2.13 and Corollary 2.14, and the second equality is due to our assumptions in the statement.
Since $\lambda$ is not of the form $\langle i_{5} \rangle$, $v_{p}(h_{\lambda\sigma})=0$ for all $\sigma$ and $\tau$ such that $\lambda \xrightarrow{\text{$\sigma$}} \tau$. Hence, 
$$J_{\mathbb{F}}(\lambda,\mu)=J_{\mathbb{C}}(\lambda,\mu).$$
By the previous theorem, $J_{\mathbb{C}}(\lambda,\mu)=0$ or $1$ when $d_{\lambda \nu}^{(e)}(v)=0$ or $v$ respectively. Therefore, $J_{\mathbb{F}}(\lambda,\mu)=J_{\mathbb{C}}(\lambda,\mu)\le 1$ and we have 
$$[S^{\lambda}_{\mathbb{C}}:D^{\mu}_{\mathbb{C}}]=J_{\mathbb{C}}(\lambda,\mu)=J_{\mathbb{F}}(\lambda,\mu)=[S^{\lambda}_{\mathbb{F}}:D^{\mu}_{\mathbb{F}}].$$
By the definition of adjustment matrices, Theorem 2.13 and Corollary 2.14,
$$[S^{\lambda}_{\mathbb{F}}:D^{\mu}_{\mathbb{F}}]=\sum\limits_{\lambda\le_{J} \sigma \le_{J} \mu} \textnormal{adj}_{\sigma \mu}[S^{\lambda}_{\mathbb{C}}:D^{\sigma}_{\mathbb{C}}]=[S^{\lambda}_{\mathbb{C}}:D^{\mu}_{\mathbb{C}}]+\textnormal{adj}_{\lambda \mu}.$$
So, $\textnormal{adj}_{\lambda \mu}=0$ as required. \qed
\end{proof}

\subsection{The row removal theorem}
Given any partition $\nu=(\nu_{1},\nu_{2},\dots,\nu_{n})$ of $n$, we define $\nu^{2}:=(\nu_{2},\nu_{3},\dots,\nu_{n})$ to be the partition of $n-\nu_{1}$ obtained from $\nu$ by removing its first component. 
\begin{theorem}
(~\cite[Theorem 6.18]{glnq}) Suppose $\lambda$ and $\mu$ are partitions of $n$, with $\mu$ $e$-regular, and that $\lambda_{1}=\mu_{1}$.
Then,
$$[S^{\lambda}:D^{\mu}]=[S^{\lambda^{2}}:D^{\mu^{2}}].$$
\end{theorem}
Let $\trianglerighteq$ denote the standard \textit{dominance order} for partitions of $n$.
\begin{corollary}
Suppose that $\lambda$ and $\mu$ are $e$-regular partitions of $n$ with $\lambda_{1}=\mu_{1}$. Then, $$\textnormal{adj}_{\lambda \mu}=\textnormal{adj}_{\lambda^{2} \mu^{2}}.$$
\end{corollary}
\begin{proof}
Fix an $e$-regular partition $\mu$ of $n$. Suppose for a contradiction that the set of $e$-regular partitions of $n$, $X_{\mu}=\{\nu:\nu_{1}=\mu_{1} ,\nu\trianglelefteq\mu, \nu\neq\mu,\textnormal{adj}_{\nu \mu}\neq\textnormal{adj}_{\nu^{2} \mu^{2}}\}$ is not empty. Let $\tau$ be a maximal element in the dominance order of $X_{\mu}$.
 \begin{align*}
[S^{\tau}_{\mathbb{F}}:D^{\mu}_{\mathbb{F}}]&=[S^{\tau}_{\mathbb{C}}:D^{\mu}_{\mathbb{C}}]+\textnormal{adj}_{\tau \mu}+\sum\limits_{\tau \triangleleft \sigma \triangleleft \mu} \textnormal{adj}_{\sigma \mu}[S^{\tau}_{\mathbb{C}}:D^{\sigma}_{\mathbb{C}}]\\
\end{align*}
On the other hand,
\begin{align*}
[S^{\tau^{2}}_{\mathbb{F}}:D^{\mu^{2}}_{\mathbb{F}}]&=[S^{\tau^{2}}_{\mathbb{C}}:D^{\mu^{2}}_{\mathbb{C}}]+\textnormal{adj}_{\tau^{2} \mu^{2}}+\sum\limits_{\tau^{2} \triangleleft \gamma \triangleleft \mu^{2}} \textnormal{adj}_{\gamma \mu^{2}}[S^{\tau^{2}}_{\mathbb{C}}:D^{\gamma}_{\mathbb{C}}]\\
\end{align*}
The function $\sigma\rightarrow\sigma^{2}$ is a bijection from the set $\{\sigma:\tau \triangleleft \sigma \triangleleft \mu\}$ to $\{\gamma:\tau^{2} \triangleleft \gamma \triangleleft \mu^{2}\}$. Moreover, $\tau \triangleleft \sigma\triangleleft \mu$ implies that $\textnormal{adj}_{\sigma \mu}=\textnormal{adj}_{\sigma^{2} \mu^{2}}$ due to the maximality of $\tau$ in $X_{\mu}$. Combined with Theorem 2.18, we get $$\textnormal{adj}_{\tau \mu}=\textnormal{adj}_{\tau^{2} \mu^{2}},$$ which is a contradiction. \qed
\end{proof}
\begin{remark}
In section 3, we only consider the principal block, $B$ of $\mathcal{H}_{5e}$. If $\lambda$ and $\mu$ are partitions in $B$, then $\lambda^{2}$ and $\mu^{2}$ must be in a block of weight at most 4 and therefore are under the purview of Theorem 2.8. If moreover, $\lambda$ and $\mu$ are $e$-regular and $\lambda_{1}=\mu_{1}$, we may apply Corollary 2.19 to conclude that $\textnormal{adj}_{\lambda \mu}=\delta_{\lambda\mu}$, the Kronecker delta.
\end{remark}

\subsection{Lowerable partitions}
The following proposition allows us to make use of the work done by Richards and Fayers for blocks of weight less than 5 (Theorem 2.8) by inducing and restricting simple modules from blocks of weight 5 to blocks of weight less than 5.
\begin{proposition}
Suppose that $\textnormal{char}(\mathbb{F})\ge5$, B is a block of  $\mathcal{H}_{n}$ of weight 5, and C is a block of $\mathcal{H}_{n-1}$ of weight less than 5. Let $\lambda$ and $\mu$ be distinct e-regular partitions lying in B such that $D^{\mu}\downarrow_{C} \not=0$, while $D^{\lambda}\downarrow_{C}$ is either zero or simple. Then $\textnormal{adj}_{\lambda \mu}=0$.
\end{proposition}
\begin{proof}
This is essentially the same as ~\cite[Proposition 2.17]{wt 4}. Let $B^{0}$ and $C^{0}$ be the blocks of $\mathcal{H}^{0}_{n}$ and $\mathcal{H}^{0}_{n-1}$ respectively corresponding to $B$ and $C$.

The modular branching rules which are characteristic-free imply that there is an $e$-regular partition $\hat{\mu}$ in $C$ such that $D^{\mu}_{\mathbb{F}}\downarrow_{C}$ is an indecomposable module with simple socle $D^{\hat{\mu}}_{\mathbb{F}}$, while $D^{\mu}_{\mathbb{C}}\downarrow_{C^{0}}$ is an indecomposable module with simple socle $D^{\hat{\mu}}_{\mathbb{C}}$. Moreover, we have $[D^{\lambda}_{\mathbb{C}}\downarrow_{C^{0}}:D^{\hat{\mu}}_{\mathbb{C}}]=0$; because $D^{\lambda}_{\mathbb{C}}\downarrow_{C^{0}}$ is either simple or zero, and if the former occurs, the modular branching rules show that it will be different from $D^{\hat{\mu}}_{\mathbb{C}}$.
  
Let $T$ be the `simple branching matrix' from $B$ to $C$, with rows indexed by $e$-regular partitions in B and columns by $e$-regular partitions in $C$, and with the $(\nu,\sigma)$-entry being the composition multiplicity $[D^{\nu}_{\mathbb{F}}\downarrow_{C}:D^{\sigma}_{\mathbb{F}}]$. Let $T^{0}$ be the simple branching matrix from $B^{0}$ to $C^{0}$ defined analogously. Using the fact that restriction is an exact functor, we have $T^{0}Z=AT$, where $Z$ and $A$ are the adjustment matrices for $C$ and $B$ respectively. $Z$ is the identity matrix by Theorem 2.8, therefore
$$T^{0}=AT.$$
Comparing the $(\lambda,\hat{\mu})$-entries of both sides yields
\begin{align*}
0&=[D^{\lambda}_{\mathbb{C}}\downarrow_{C^{0}}:D^{\hat{\mu}}_{\mathbb{C}}]\\
&=\sum\limits_{\nu}\textnormal{adj}_{\lambda \nu}[D^{\nu}_{\mathbb{F}}\downarrow_{C}:D^{\hat{\mu}}_{\mathbb{F}}]\\
&=\textnormal{adj}_{\lambda \mu}[D^{\mu}_{\mathbb{F}}\downarrow_{C}:D^{\hat{\mu}}_{\mathbb{F}}]+\sum\limits_{\nu \not= \mu}\textnormal{adj}_{\lambda \nu}[D^{\nu}_{\mathbb{F}}\downarrow_{C}:D^{\hat{\mu}}_{\mathbb{F}}].
\end{align*}
Since every term of the sum is non-negative and $[D^{\mu}_{\mathbb{F}}\downarrow_{C}:D^{\hat{\mu}}_{\mathbb{F}}]>0$, we conclude that $\textnormal{adj}_{\lambda \mu}=0$. \qed
\end{proof}
\begin{definition}
If $\lambda$ and $\mu$ satisfy the conditions of the proposition above, we say that $(\lambda, \mu)$ is $lowerable$.
\end{definition}

\section{The principal block of $\mathcal{H}_{5e}$}
Let $B$ be the principal block of $\mathcal{H}_{5e}$.
\begin{lemma}
If $\mu$ is an $e$-regular partition in B, then there is some block C of $\mathcal{H}_{5e-1}$ of weight less than 5 such that $D^{\mu}\downarrow_{C} \not=0$.
\end{lemma}

\begin{proof}
Since $\mathcal{H}_{5e-1}$ is a unital subalgebra of $\mathcal{H}_{5e}$, we have $D^{\mu}\downarrow_{\mathcal{H}_{5e-1}} \not=0$; in particular,  $D^{\mu}\downarrow_{C} \not=0$ for some block C of  $\mathcal{H}_{5e-1}$. Clearly, every block of  $\mathcal{H}_{5e-1}$ has weight less than 5, so the result follows. \qed
\end{proof}

\begin{corollary}
Suppose $\lambda$ and $\mu$ are e-regular partitions in B. If there is no block C of $\mathcal{H}_{5e-1}$ such that $D^{\lambda}\downarrow_{C}$ is reducible, then $\textnormal{adj}_{\lambda\mu}=0$.
\end{corollary}
\begin{proof}
Suppose $\lambda$ and $\mu$ are as in the statement. By the previous lemma, there is a block C of $\mathcal{H}_{5e-1}$ such that $D^{\mu}\downarrow_{C} \not=0$. By assumption, $D^{\lambda}\downarrow_{C}$ is zero or simple. Therefore $(\lambda,\mu)$ is lowerable and Proposition 2.20 implies that $\textnormal{adj}_{\lambda\mu}=0$. \qed
\end{proof}
The only $e$-regular partitions $\lambda$ in $B$ such that $D^{\lambda}$ is reducible after restricting to some block of $\mathcal{H}_{5e-1}$ are:
\begin{align*}
\bullet& \langle i_{3,2} \rangle, &  i \in \{1,2,\dots,e-1\}.\\
\bullet& \langle i_{2^{2}},j \rangle, &  1\le i \le j-1, e\ge3\\
\bullet& \langle j,i_{2^{2}} \rangle, &  j\le i-2, e\ge3.
\end{align*}
In light of the corollary, we need only consider these rows of the adjustment matrix in order to prove James's Conjecture for the block $B$.\\
The weight $4$ block $C$ with the $\langle 5^{i-1},6,4,5^{e-i-1} \rangle$ notation is the only block of $\mathcal{H}_{5e-1}$ such that $D^{\lambda}\downarrow_{C}$ is reducible. Hence, if $D^{\mu}\downarrow_{D}\not= 0$ for some block $D\not= C$, then $(\lambda,\mu)$ would be lowerable and by Proposition 2.20, $\textnormal{adj}_{\lambda \mu}=0$. Therefore, we may assume that  $D^{\mu}\downarrow_{D}= 0$ for every block $D$ of $\mathcal{H}_{5e-1}$ other than $C$.\\

\subsection{$\lambda=\langle i_{3,2} \rangle$}
By Lemma 2.15 and Corollary 2.19, $\textnormal{adj}_{\lambda \mu}\neq 0\Rightarrow \mu_{1}>\lambda_{1}$ and $\mu>_{P}\lambda$. Those $\mu$ satisfying $\mu_{1}>\lambda_{1}$ and $\mu>_{P}\lambda$ are:
\begin{itemize}
\item$\langle i_{5}\rangle$
\item$\langle i_{4},i+1\rangle$
\item$\langle 0,i_{4} \rangle$
\end{itemize}
Proposition 2.10 gives $\textnormal{adj}_{\lambda \mu}=\textnormal{adj}_{\lambda^{\diamond} \mu^{\diamond}}$. Therefore, we may also assume that $\mu^{\diamond}_{1}> \lambda^{\diamond}_{1}$ and $\mu^{\diamond}>_{P} \lambda^{\diamond}$. We calculate $\lambda^{\diamond}$ and $\mu^{\diamond}$ for all of the pairs above and list the pairs $(\lambda,\mu)$ satisfying these 2 conditions in the following table. 

\begin{table}[H]
\begin{tabular}{|p{1cm}|p{2.5cm}|p{2.5cm}|p{2.5cm}|p{2.5cm}|p{1.3cm}|p{1.3cm}|} \hline
$e$ & $\lambda$ & $\mu$ & $\lambda^{\diamond}$ & $\mu^{\diamond}$ & $d^{(e)}_{\lambda\mu}(v)$ & $d^{(e)}_{\lambda^{\diamond}\mu^{\diamond}}(v)$\\ \hline
2&$\langle1_{3,2}\rangle$&$\langle1_{5}\rangle$&$\langle1_{3,2}\rangle$&$\langle1_{5}\rangle$&$0$&$0$\\ \hline
2&$\langle1_{3,2}\rangle$&$\langle0,1_{4}\rangle$&$\langle1_{3,2}\rangle$&$\langle0,1_{4}\rangle$&$v$&$v$\\ \hline
3&$\langle2_{3,2}\rangle$&$\langle2_{5}\rangle$&$\langle1_{2^{2}},2\rangle$&$\langle1_{3},2_{2} \rangle$&0&\\ \hline
3&$\langle2_{3,2}\rangle$&$\langle0,2_{4}\rangle$&$\langle1_{2^{2}},2\rangle$&$\langle0,1_{2},2_{2} \rangle$&$v$&\\ \hline
4&$\langle3_{3,2}\rangle$&$\langle3_{5}\rangle$&$\langle1_{2^{2}},2\rangle$& $\langle 1_{2},2_{2},3 \rangle$ &$v$&\\ \hline
4&$\langle2_{3,2}\rangle$&$\langle0,2_{4}\rangle$&$\langle0,2_{2^{2}}\rangle$&$\langle0,2_{4}\rangle$&&0\\ \hline
\end{tabular}
\end{table}
We observe that either $d^{(e)}_{\lambda\mu}(v)\in\{0,v\}$ or $d^{(e)}_{\lambda^{\diamond}\mu^{\diamond}}(v)\in\{0,v\}$ (we only need one of them to hold, which is why some entries of the table are left empty) in all cases. Therefore, Corollary 2.17 applies and we conclude that $\textnormal{adj}_{\lambda\mu}=0$ for every $e$-regular $\mu$ when $\lambda$ is of the form $\langle i_{3,2}\rangle$.

\subsection{$\lambda=\langle i_{2^{2}},j \rangle$ or $\langle j,i_{2^{2}} \rangle$}
By combining Proposition 2,10, Lemma 2.15 and Corollary 2.19 as seen in the last subsection, it is only possible for $\textnormal{adj}_{\lambda\mu}\neq0$ when $\lambda\neq\mu$, $\lambda_{1}<\mu_{1}$, $\lambda^{\diamond}_{1}<\mu^{\diamond}_{1}$, $\lambda<_{P}\mu$ and $\lambda^{\diamond}<_{P}\mu^{\diamond}$. Moreover, we may also exclude those cases where $\lambda^{\diamond}$ is of the form $\langle i_{3,2}\rangle$ as this has been dealt with in the previous subsection. We list all the pairs $(\lambda,\mu)$ satisfying these conditions in the table below.
\begin{table}[H]
\begin{tabular}{|p{1cm}|p{2.5cm}|p{2.5cm}|p{2.5cm}|p{2.5cm}|p{1.3cm}|p{1.3cm}|} \hline
$e$ & $\lambda$ & $\mu$ & $\lambda^{\diamond}$ & $\mu^{\diamond}$ & $d^{(e)}_{\lambda\mu}(v)$ & $d^{(e)}_{\lambda^{\diamond}\mu^{\diamond}}(v)$\\ \hline
6&$\langle3_{2^{2}},5\rangle$&$\langle3_{2},4_{2},5\rangle$&$\langle1,3_{2^{2}}\rangle$&$\langle3_{2},4_{2},5\rangle$&$0$&\\ \hline
5&$\langle3_{2^{2}},4\rangle$&$\langle3_{3},4_{2}\rangle$& $\langle2_{2^{2}},3 \rangle$ &$\langle 2_{2},3_{2},4 \rangle$&0&\\ \hline
5&$\langle2_{2^{2}},4\rangle$&$\langle2_{2},3_{2},4\rangle$&$\langle1,3_{2^{2}} \rangle$& $\langle3_{3},4_{2} \rangle$&$v$&\\ \hline
5&$\langle1_{2^{2}},4\rangle$&$\langle1_{2},2_{2},3\rangle$&$\langle 1,4_{2^{2}} \rangle$& $\langle0_{3},4_{2} \rangle$&$0$&\\ \hline
5&$\langle1_{2^{2}},3\rangle$&$\langle1_{2},2_{2},3\rangle$& $\langle2,4_{2^{2}} \rangle$& $\langle0_{3},4_{2} \rangle$&$v$&\\ \hline
4&$\langle1_{2^{2}},3\rangle$&$\langle1_{3},2_{2}\rangle$&$\langle1,3_{2^{2}}\rangle$&$\langle0_{3},2_{2}\rangle$&$v$&\\ \hline
4&$\langle2_{2^{2}},3\rangle$&$\langle2_{3},3_{2}\rangle$&$\langle2_{2^{2}},3\rangle$&$\langle2_{3},3_{2}\rangle$&$v$&\\ \hline
4&$\langle2_{2^{2}},3\rangle$&$\langle0,2_{2},3_{2}\rangle$&$\langle2_{2^{2}},3\rangle$&$\langle0,2_{2},3_{2}\rangle$&$v$&\\ \hline
4&$\langle1_{2^{2}},3\rangle$&$\langle1_{2},2_{2},3\rangle$&$\langle1,3_{2^{2}}\rangle$&$\langle3_{5}\rangle$&$v^{2}$&$3v^{2}$\\ \hline
$e\ge4$&$\langle1,(e-1)_{2^{2}}\rangle$&$\langle0,1,(e-1)_{3}\rangle$&$\langle1_{2^{2}},e-1\rangle$&$\langle0,1_{3},2\rangle$&&\\ \hline
\end{tabular}
\end{table}
We apply Corollary 2.17 successively to conclude that $\textnormal{adj}_{\lambda\mu}=\delta_{\lambda\mu}$ except when:
\begin{align*}
\bullet& (\lambda,\mu)=(\langle1_{2^{2}},3\rangle,\langle1_{2},2_{2},3\rangle)=(\langle1,3_{2^{2}}\rangle^{\diamond},\langle3_{5}\rangle^{\diamond}), &    e=4\\
\bullet& (\lambda,\mu)=(\langle1,(e-1)_{2^{2}}\rangle,\langle0,1,(e-1)_{3}\rangle)=(\langle1_{2^{2}},e-1\rangle^{\diamond},\langle0,1_{3},2\rangle^{\diamond}), &    e\ge4
\end{align*}
The author has not been able to calculate either $\textnormal{adj}_{\lambda\mu}$ or $\textnormal{adj}_{\lambda^{\diamond}\mu^{\diamond}}$ for the first case. This may be explored in a future work. For the rest of the paper, we deal with the case: $$(\lambda,\mu)=(\langle1,(e-1)_{2^{2}}\rangle,\langle0,1,(e-1)_{3}\rangle), e\ge4$$
Let $D$ be the weight 4 block with the $\langle 5^{e-2},6,4\rangle$ notation. We define the partitions $\tilde{\mu}$, $\tilde{\lambda_{0}}$ and $\tilde{\lambda_{1}}$ by the following abacus diagrams:

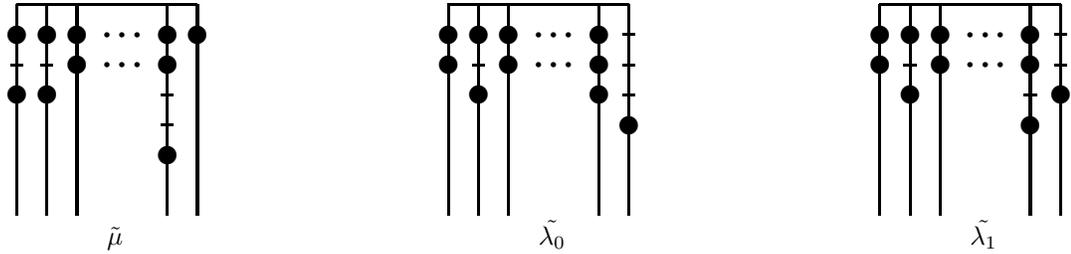
\setlength{\unitlength}{8mm}
\begin{picture}(4,5)(0,-4)
\thicklines
\put(1.5,-3.5){{\footnotesize $\tilde{\mu}$}}
\put(0,0){\circle*{0.3}}
\put(-0.1,-0.5){\line(1,0){0.2}}
\put(0,-1){\circle*{0.3}}
\put(0,0.5){\line(0,-1){3.5}}
\put(0.5,0){\circle*{0.3}}
\put(0.4,-0.5){\line(1,0){0.2}}
\put(0.5,-1){\circle*{0.3}}
\put(0.5,0.5){\line(0,-1){3.5}}
\put(1,0){\circle*{0.3}}
\put(1,-0.5){\circle*{0.3}}
\put(1,0.5){\line(0,-1){3.5}}
\put(1.5,0){\circle*{0.1}}
\put(1.75,0){\circle*{0.1}}
\put(2,0){\circle*{0.1}}
\put(1.5,-0.5){\circle*{0.1}}
\put(2,-0.5){\circle*{0.1}}
\put(1.75,-0.5){\circle*{0.1}}
\put(2.5,0){\circle*{0.3}}
\put(2.5,-0.5){\circle*{0.3}}
\put(2.5,-2){\circle*{0.3}}
\put(3,0){\circle*{0.3}}
\put(2.4,-1){\line(1,0){0.2}}
\put(2.4,-1.5){\line(1,0){0.2}}
\put(2.5,0.5){\line(0,-1){3.5}}
\put(3,0.5){\line(0,-1){3.5}}
\put(0,0.5){\line(1,0){3}}
\end{picture}
\begin{picture}(4,5)(-3,-4)
\thicklines
\put(1.5,-3.5){{\footnotesize $\tilde{\lambda_{0}}$}}
\put(0,0){\circle*{0.3}}
\put(0,-0.5){\circle*{0.3}}
\put(0,0.5){\line(0,-1){3.5}}
\put(0.5,0){\circle*{0.3}}
\put(0.4,-0.5){\line(1,0){0.2}}
\put(0.5,-1){\circle*{0.3}}
\put(0.5,0.5){\line(0,-1){3.5}}
\put(1,0){\circle*{0.3}}
\put(1,-0.5){\circle*{0.3}}
\put(1,0.5){\line(0,-1){3.5}}
\put(1.5,0){\circle*{0.1}}
\put(1.75,0){\circle*{0.1}}
\put(2,0){\circle*{0.1}}
\put(1.5,-0.5){\circle*{0.1}}
\put(2,-0.5){\circle*{0.1}}
\put(1.75,-0.5){\circle*{0.1}}
\put(2.5,0){\circle*{0.3}}
\put(2.5,-0.5){\circle*{0.3}}
\put(2.5,-1){\circle*{0.3}}
\put(3,-1.5){\circle*{0.3}}
\put(2.9,-1){\line(1,0){0.2}}
\put(2.9,-0.5){\line(1,0){0.2}}
\put(2.9,0){\line(1,0){0.2}}
\put(2.5,0.5){\line(0,-1){3.5}}
\put(3,0.5){\line(0,-1){3.5}}
\put(0,0.5){\line(1,0){3}}
\end{picture}
\begin{picture}(4,5)(-6,-4)
\thicklines
\put(1.5,-3.5){{\footnotesize $\tilde{\lambda_{1}}$}}
\put(0,0){\circle*{0.3}}
\put(0,-0.5){\circle*{0.3}}
\put(0,0.5){\line(0,-1){3.5}}
\put(0.5,0){\circle*{0.3}}
\put(0.4,-0.5){\line(1,0){0.2}}
\put(0.5,-1){\circle*{0.3}}
\put(0.5,0.5){\line(0,-1){3.5}}
\put(1,0){\circle*{0.3}}
\put(1,-0.5){\circle*{0.3}}
\put(1,0.5){\line(0,-1){3.5}}
\put(1.5,0){\circle*{0.1}}
\put(1.75,0){\circle*{0.1}}
\put(2,0){\circle*{0.1}}
\put(1.5,-0.5){\circle*{0.1}}
\put(2,-0.5){\circle*{0.1}}
\put(1.75,-0.5){\circle*{0.1}}
\put(2.5,0){\circle*{0.3}}
\put(2.5,-0.5){\circle*{0.3}}
\put(2.5,-1.5){\circle*{0.3}}
\put(3,-1){\circle*{0.3}}
\put(2.4,-1){\line(1,0){0.2}}
\put(2.9,-0.5){\line(1,0){0.2}}
\put(2.9,0){\line(1,0){0.2}}
\put(2.5,0.5){\line(0,-1){3.5}}
\put(3,0.5){\line(0,-1){3.5}}
\put(0,0.5){\line(1,0){3}}
\end{picture}
\\
By the modular branching rules, $D^{\tilde{\mu}}$ is the only simple module in block $D$ that upon induction to block $B$ has $D^{\mu}$ appearing in its head. Similarly, the only Specht modules in $D$ that upon induction to $B$ have a filtration with a factor of $S^{\lambda}$ are $S^{\tilde{\lambda_{0}}}$ and $S^{\tilde{\lambda_{1}}}$.
By Proposition 2.5, we may write $f_{e-1}(G(\tilde{\mu}))$ in the form 
\begin{equation}
f_{e-1}(G(\tilde{\mu}))=\sum\limits_{\nu} a_{\nu}(v)G(\nu),
\end{equation}
where $a_{\nu}(v)\in\mathbb{N}_{0}[v+v^{-1}]$ for all $\nu$.
\\
If we manage to show that $a_{\lambda}(v)=0$, then we have that 
\begin{equation}
P_{\mathbb{C}}^{\tilde{\mu}}\uparrow_{D}^{B}\cong\bigoplus_{\nu\neq\lambda}a_{\nu}(1)P_{\mathbb{C}}^{\nu}
\end{equation}
since $f_{e-1}$ acts like ($e-1$)-Ind when $v=1$.
Since James's Conjecture holds for blocks of weight four, $$\textbf{e}_{\mathbb{C}}([\![P^{\tilde{\mu}}_{\mathbb{C}}\uparrow^{B}_{D}]\!]^{p})=\textbf{e}_{\mathbb{F}}([\![P^{\tilde{\mu}}_{\mathbb{F}}\uparrow^{B}_{D}]\!]^{p}).$$
By equation (2), the left-hand side is $\sum\limits_{\nu\neq\lambda}a_{\nu}(1)\textbf{e}_{\mathbb{C}}([\![P_{\mathbb{C}}^{\nu}]\!]^{p})$.
On the other hand, the right-hand side contains the term $\textbf{e}_{\mathbb{F}}([\![P_{\mathbb{F}}^{\mu}]\!]^{p})=\textbf{e}_{\mathbb{C}}([\![P_{\mathbb{C}}^{\mu}]\!]^{p})+\textnormal{adj}_{\lambda\mu}\textbf{e}_{\mathbb{C}}([\![P_{\mathbb{C}}^{\lambda}]\!]^{p})$.
Since $\textbf{e}_{\mathbb{C}}$ is injective and the set of $[\![P_{\mathbb{C}}^{\nu}]\!]^{p}$ as $\nu$ runs over all $e$-regular partitions of $n$ is a linearly independent set in $\mathcal{G}^{p}_{n}(\mathbb{C})$, $\textnormal{adj}_{\lambda\mu}$ must be zero. 

\begin{proposition}
$a_{\lambda}(v)=0$.
\end{proposition}
\begin{proof}
By the definition of $v$-decomposition numbers, $G(\tilde{\mu})=\sum\limits_{\tilde{\nu}\in D}d^{(e)}_{\tilde{\nu}\tilde{\mu}}(v)s(\tilde{\nu})$.
So, 
\begin{equation}
f_{e-1}(G(\tilde{\mu}))=\sum\limits_{\tilde{\nu}\in D}d^{(e)}_{\tilde{\nu}\tilde{\mu}}(v)f_{e-1}(s(\tilde{\nu})).
\end{equation}
In this sum, only the terms $\tilde{\nu}=\tilde{\lambda_{0}}$ and $\tilde{\nu}=\tilde{\lambda_{1}}$ may contribute to the coefficients of $s(\lambda)$.
Let $D^{(i)}$ be the weight 4 block with the $\langle 5^{i},6,5^{e-i-2},4 \rangle$ notation and $E$ be the weight 2 block with the $\langle 4,5^{e-2},6 \rangle$ notation. We define the partitions $\hat{\mu}$ and $\hat{\lambda_{0}}$ by the abacus diagrams in figure 1 and figure 2. Modular branching rules yield
$$S^{\tilde{\lambda_{0}}}\downarrow^{D}_{D^{(e-3)}}\downarrow^{D^{(e-3)}}_{D^{(e-4)}}\cdots \downarrow^{D^{(3)}}_{D^{(2)}}\downarrow^{D^{(2)}}_{D^{(1)}}\downarrow^{D^{(1)}}_{D^{(0)}}\downarrow^{D^{(0)}}_{E}\sim (S^{\hat{\lambda_{0}}})^{2},$$
$$D^{\tilde{\mu}}\downarrow^{D}_{D^{(e-3)}}\downarrow^{D^{(e-3)}}_{D^{(e-4)}}\cdots \downarrow^{D^{(3)}}_{D^{(2)}}\downarrow^{D^{(2)}}_{D^{(1)}}\downarrow^{D^{(1)}}_{D^{(0)}}\downarrow^{D^{(0)}}_{E}\cong (D^{\hat{\mu}})^{\oplus2}.$$
Using the product order, we see that $\hat{\mu}\le_{P} \hat{\lambda_{0}}$.
Therefore, $$[S^{\hat{\lambda_{0}}}:D^{\hat{\mu}}]=0 \Rightarrow [S^{\tilde{\lambda_{0}}}:D^{\tilde{\mu}}]=0.$$
Hence, $d^{(e)}_{\tilde{\lambda_{0}}\tilde{\mu}}(v)=0$ and we can conclude that the term $\tilde{\nu}=\tilde{\lambda_{0}}$ in (3) has no contribution. Let us now focus our attention on the term $\tilde{\nu}=\tilde{\lambda_{1}}$.
Since $f_{e-1}(s(\tilde{\lambda_{1}}))=s(\lambda)$, the coefficient of $s(\lambda)$ in (3) must be $d^{(e)}_{\tilde{\lambda_{1}}\tilde{\mu}}(v)\in v\mathbb{N}_{0}[v]$. On the other hand, the coefficient of $s(\lambda)$ in (1) is $a_{\lambda}(v)+\sum\limits_{\lambda\le_{J}\nu}a_{\nu}(v)d^{(e)}_{\lambda\nu}(v)$. If $a_{\lambda}(v)\neq 0$, then the coefficient of $s(\lambda)$ in (1) would include either constant terms or negative powers of $v$ since $a_{\nu}(v)\in\mathbb{N}_{0}[v+v^{-1}]$ for all $\nu$. This is a contradiction.

Therefore, $a_{\lambda}(v)= 0$ and $\textnormal{adj}_{\lambda\mu}=0$. \qed
\end{proof}

This ends the proof of Theorem 2.9.

\pagebreak

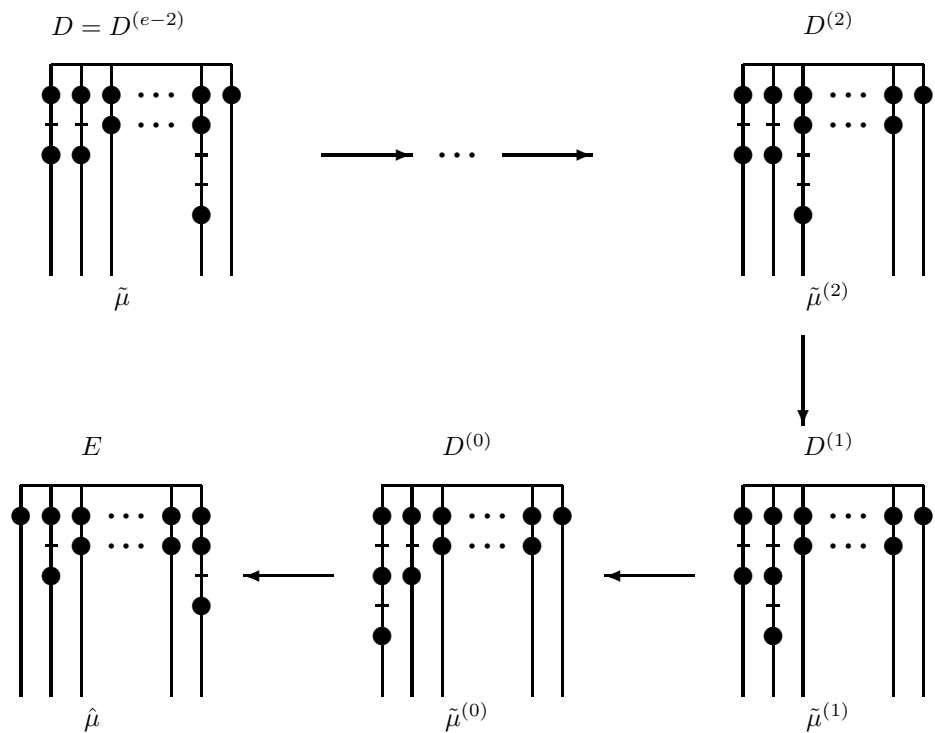
\begin{picture}(12,12)(-6,-14)
\thicklines
\put(1.5,2){{\footnotesize Figure 1}}
\put(-4.5,1){{\footnotesize $D=D^{(e-2)}$}}
\put(-3.45,-3.5){{\footnotesize $\tilde{\lambda_{0}}$}}
\put(-4.5,0){\circle*{0.3}}
\put(-4.5,-0.5){\circle*{0.3}}
\put(-4.5,0.5){\line(0,-1){3.5}}
\put(-4,0){\circle*{0.3}}
\put(-4.1,-0.5){\line(1,0){0.2}}
\put(-4,-1){\circle*{0.3}}
\put(-4,0.5){\line(0,-1){3.5}}
\put(-3,0){\circle*{0.3}}
\put(-3,-0.5){\circle*{0.3}}
\put(-3,0.5){\line(0,-1){3.5}}
\put(-3,0){\circle*{0.1}}
\put(-2.75,0){\circle*{0.1}}
\put(-2.5,0){\circle*{0.1}}
\put(-3,-0.5){\circle*{0.1}}
\put(-2.5,-0.5){\circle*{0.1}}
\put(-2.75,-0.5){\circle*{0.1}}
\put(-2,0){\circle*{0.3}}
\put(-2,-0.5){\circle*{0.3}}
\put(-2,-1){\circle*{0.3}}
\put(-1.5,-1.5){\circle*{0.3}}
\put(-1.6,-1){\line(1,0){0.2}}
\put(-1.6,-0.5){\line(1,0){0.2}}
\put(-1.6,0){\line(1,0){0.2}}
\put(-2,0.5){\line(0,-1){3.5}}
\put(-1.5,0.5){\line(0,-1){3.5}}
\put(-4.5,0.5){\line(1,0){3}}

\put(0, -1){\vector(1, 0){1.5}}
\put(2,-1){\circle*{0.1}}
\put(2.25,-1){\circle*{0.1}}
\put(2.5,-1){\circle*{0.1}}
\put(3, -1){\vector(1, 0){1.5}}

\put(8,1){{\footnotesize $D^{(2)}$}}
\put(8.05,-3.5){{\footnotesize $\tilde{\lambda_{0}}^{(2)}$}}
\put(7,0){\circle*{0.3}}
\put(7,-0.5){\circle*{0.3}}
\put(7,0.5){\line(0,-1){3.5}}
\put(7.5,0){\circle*{0.3}}
\put(7.4,-0.5){\line(1,0){0.2}}
\put(7.5,-1){\circle*{0.3}}
\put(7.5,0.5){\line(0,-1){3.5}}
\put(8,0){\circle*{0.3}}
\put(8,-0.5){\circle*{0.3}}
\put(8,0.5){\line(0,-1){3.5}}
\put(8.5,0){\circle*{0.1}}
\put(8.75,0){\circle*{0.1}}
\put(9,0){\circle*{0.1}}
\put(8.5,-0.5){\circle*{0.1}}
\put(9,-0.5){\circle*{0.1}}
\put(8.75,-0.5){\circle*{0.1}}
\put(9.5,0){\circle*{0.3}}
\put(9.5,-0.5){\circle*{0.3}}
\put(8,-1){\circle*{0.3}}
\put(10,-1.5){\circle*{0.3}}
\put(9.9,-1){\line(1,0){0.2}}
\put(9.9,-0.5){\line(1,0){0.2}}
\put(9.9,0){\line(1,0){0.2}}
\put(9.5,0.5){\line(0,-1){3.5}}
\put(10,0.5){\line(0,-1){3.5}}
\put(7,0.5){\line(1,0){3}}

\put(8, -4){\vector(0, -1){1.5}}

\put(8,-6){{\footnotesize $D^{(1)}$}}
\put(8.05,-10.5){{\footnotesize $\tilde{\lambda_{0}}^{(1)}$}}
\put(7,-7){\circle*{0.3}}
\put(7,-7.5){\circle*{0.3}}
\put(7,-6.5){\line(0,-1){3.5}}
\put(7.5,-7){\circle*{0.3}}
\put(7.4,-7.5){\line(1,0){0.2}}
\put(7.5,-8){\circle*{0.3}}
\put(7.5,-6.5){\line(0,-1){3.5}}
\put(8,-7){\circle*{0.3}}
\put(7.5,-7.5){\circle*{0.3}}
\put(7.9,-7.5){\line(1,0){0.2}}
\put(8,-6.5){\line(0,-1){3.5}}
\put(8.5,-7){\circle*{0.1}}
\put(8.75,-7){\circle*{0.1}}
\put(9,-7){\circle*{0.1}}
\put(8.5,-7.5){\circle*{0.1}}
\put(9,-7.5){\circle*{0.1}}
\put(8.75,-7.5){\circle*{0.1}}
\put(9.5,-7){\circle*{0.3}}
\put(9.5,-7.5){\circle*{0.3}}
\put(8,-8){\circle*{0.3}}
\put(10,-8.5){\circle*{0.3}}
\put(9.9,-8){\line(1,0){0.2}}
\put(9.9,-7.5){\line(1,0){0.2}}
\put(9.9,-7){\line(1,0){0.2}}
\put(9.5,-6.5){\line(0,-1){3.5}}
\put(10,-6.5){\line(0,-1){3.5}}
\put(7,-6.5){\line(1,0){3}}

\put(6.2, -8){\vector(-1,0){1.5}}

\put(2,-6){{\footnotesize $D^{(0)}$}}
\put(2.05,-10.5){{\footnotesize $\tilde{\lambda_{0}}^{(0)}$}}
\put(1,-7){\circle*{0.3}}
\put(1,-7.5){\circle*{0.3}}
\put(1,-6.5){\line(0,-1){3.5}}
\put(1.5,-7){\circle*{0.3}}
\put(1.4,-7.5){\line(1,0){0.2}}
\put(1,-8){\circle*{0.3}}
\put(1.5,-6.5){\line(0,-1){3.5}}
\put(2,-7){\circle*{0.3}}
\put(1.5,-7.5){\circle*{0.3}}
\put(1.9,-7.5){\line(1,0){0.2}}
\put(2,-6.5){\line(0,-1){3.5}}
\put(2.5,-7){\circle*{0.1}}
\put(2.75,-7){\circle*{0.1}}
\put(3,-7){\circle*{0.1}}
\put(2.5,-7.5){\circle*{0.1}}
\put(3,-7.5){\circle*{0.1}}
\put(2.75,-7.5){\circle*{0.1}}
\put(3.5,-7){\circle*{0.3}}
\put(3.5,-7.5){\circle*{0.3}}
\put(2,-8){\circle*{0.3}}
\put(4,-8.5){\circle*{0.3}}
\put(3.9,-8){\line(1,0){0.2}}
\put(3.9,-7.5){\line(1,0){0.2}}
\put(3.9,-7){\line(1,0){0.2}}
\put(3.5,-6.5){\line(0,-1){3.5}}
\put(4,-6.5){\line(0,-1){3.5}}
\put(1,-6.5){\line(1,0){3}}

\put(0.2, -8){\vector(-1, 0){1.5}}

\put(-4,-6){{\footnotesize $E$}}
\put(-3.95,-10.5){{\footnotesize $\hat{\lambda_{0}}$}}
\put(-5,-7){\circle*{0.3}}
\put(-2,-7){\circle*{0.3}}
\put(-5,-6.5){\line(0,-1){3.5}}
\put(-4.5,-7){\circle*{0.3}}
\put(-4.6,-7.5){\line(1,0){0.2}}
\put(-2,-7.5){\circle*{0.3}}
\put(-4.5,-6.5){\line(0,-1){3.5}}
\put(-4,-7){\circle*{0.3}}
\put(-4.5,-7.5){\circle*{0.3}}
\put(-4.1,-7.5){\line(1,0){0.2}}
\put(-4,-6.5){\line(0,-1){3.5}}
\put(-3.5,-7){\circle*{0.1}}
\put(-3.25,-7){\circle*{0.1}}
\put(-3,-7){\circle*{0.1}}
\put(-3.5,-7.5){\circle*{0.1}}
\put(-3,-7.5){\circle*{0.1}}
\put(-3.25,-7.5){\circle*{0.1}}
\put(-2.5,-7){\circle*{0.3}}
\put(-2.5,-7.5){\circle*{0.3}}
\put(-4,-8){\circle*{0.3}}
\put(-2,-8.5){\circle*{0.3}}
\put(-2.1,-8){\line(1,0){0.2}}
\put(-2.1,-7.5){\line(1,0){0.2}}
\put(-2.1,-7){\line(1,0){0.2}}
\put(-2.5,-6.5){\line(0,-1){3.5}}
\put(-2,-6.5){\line(0,-1){3.5}}
\put(-5,-6.5){\line(1,0){3}}
\end{picture}

\begin{picture}(12,12)(-6,-12)
\thicklines
\put(1.5,2){{\footnotesize Figure 2}}
\put(-4.5,1){{\footnotesize $D=D^{(e-2)}$}}
\put(-3.45,-3.5){{\footnotesize $\tilde{\mu}$}}
\put(-4.5,0){\circle*{0.3}}
\put(-4.6,-0.5){\line(1,0){0.2}}
\put(-4.5,-1){\circle*{0.3}}
\put(-4.5,0.5){\line(0,-1){3.5}}
\put(-4,0){\circle*{0.3}}
\put(-4.1,-0.5){\line(1,0){0.2}}
\put(-4.5,-1){\circle*{0.3}}
\put(-4,-1){\circle*{0.3}}
\put(-4,0.5){\line(0,-1){3.5}}
\put(-3.5,0){\circle*{0.3}}
\put(-3.5,-0.5){\circle*{0.3}}
\put(-3.5,0.5){\line(0,-1){3.5}}
\put(-3,0){\circle*{0.1}}
\put(-2.75,0){\circle*{0.1}}
\put(-2.5,0){\circle*{0.1}}
\put(-3,-0.5){\circle*{0.1}}
\put(-2.5,-0.5){\circle*{0.1}}
\put(-2.75,-0.5){\circle*{0.1}}
\put(-2,0){\circle*{0.3}}
\put(-2,-0.5){\circle*{0.3}}
\put(-2,-2){\circle*{0.3}}
\put(-1.5,0){\circle*{0.3}}
\put(-2.1,-1){\line(1,0){0.2}}
\put(-2.1,-1.5){\line(1,0){0.2}}
\put(-2,0.5){\line(0,-1){3.5}}
\put(-1.5,0.5){\line(0,-1){3.5}}
\put(-4.5,0.5){\line(1,0){3}}

\put(0, -1){\vector(1, 0){1.5}}
\put(2,-1){\circle*{0.1}}
\put(2.25,-1){\circle*{0.1}}
\put(2.5,-1){\circle*{0.1}}
\put(3, -1){\vector(1, 0){1.5}}

\put(8,1){{\footnotesize $D^{(2)}$}}
\put(8.05,-3.5){{\footnotesize $\tilde{\mu}^{(2)}$}}
\put(7,0){\circle*{0.3}}
\put(6.9,-0.5){\line(1,0){0.2}}
\put(7,-1){\circle*{0.3}}
\put(7,0.5){\line(0,-1){3.5}}
\put(7.5,0){\circle*{0.3}}
\put(7.4,-0.5){\line(1,0){0.2}}
\put(7.5,-1){\circle*{0.3}}
\put(7.5,0.5){\line(0,-1){3.5}}
\put(8,0){\circle*{0.3}}
\put(8,-0.5){\circle*{0.3}}
\put(8,0.5){\line(0,-1){3.5}}
\put(8.5,0){\circle*{0.1}}
\put(8.75,0){\circle*{0.1}}
\put(9,0){\circle*{0.1}}
\put(8.5,-0.5){\circle*{0.1}}
\put(9,-0.5){\circle*{0.1}}
\put(8.75,-0.5){\circle*{0.1}}
\put(9.5,0){\circle*{0.3}}
\put(9.5,-0.5){\circle*{0.3}}
\put(8,-2){\circle*{0.3}}
\put(10,0){\circle*{0.3}}
\put(7.9,-1){\line(1,0){0.2}}
\put(7.9,-1.5){\line(1,0){0.2}}
\put(9.5,0.5){\line(0,-1){3.5}}
\put(10,0.5){\line(0,-1){3.5}}
\put(7,0.5){\line(1,0){3}}

\put(8, -4){\vector(0, -1){1.5}}

\put(8,-6){{\footnotesize $D^{(1)}$}}
\put(8.05,-10.5){{\footnotesize $\tilde{\mu}^{(1)}$}}
\put(7,-7){\circle*{0.3}}
\put(6.9,-7.5){\line(1,0){0.2}}
\put(7,-8){\circle*{0.3}}
\put(7,-6.5){\line(0,-1){3.5}}
\put(7.5,-7){\circle*{0.3}}
\put(7.4,-7.5){\line(1,0){0.2}}
\put(7.5,-8){\circle*{0.3}}
\put(7.5,-6.5){\line(0,-1){3.5}}
\put(8,-7){\circle*{0.3}}
\put(8,-7.5){\circle*{0.3}}
\put(8,-6.5){\line(0,-1){3.5}}
\put(8.5,-7){\circle*{0.1}}
\put(8.75,-7){\circle*{0.1}}
\put(9,-7){\circle*{0.1}}
\put(8.5,-7.5){\circle*{0.1}}
\put(9,-7.5){\circle*{0.1}}
\put(8.75,-7.5){\circle*{0.1}}
\put(9.5,-7){\circle*{0.3}}
\put(9.5,-7.5){\circle*{0.3}}
\put(7.5,-9){\circle*{0.3}}
\put(10,-7){\circle*{0.3}}
\put(7.4,-8.5){\line(1,0){0.2}}
\put(9.5,-6.5){\line(0,-1){3.5}}
\put(10,-6.5){\line(0,-1){3.5}}
\put(7,-6.5){\line(1,0){3}}

\put(6.2, -8){\vector(-1, 0){1.5}}

\put(2,-6){{\footnotesize $D^{(0)}$}}
\put(2.05,-10.5){{\footnotesize $\tilde{\mu}^{(0)}$}}
\put(1,-7){\circle*{0.3}}
\put(0.9,-7.5){\line(1,0){0.2}}
\put(1,-8){\circle*{0.3}}
\put(1,-6.5){\line(0,-1){3.5}}
\put(1.5,-7){\circle*{0.3}}
\put(1.4,-7.5){\line(1,0){0.2}}
\put(1.5,-8){\circle*{0.3}}
\put(1.5,-6.5){\line(0,-1){3.5}}
\put(2,-7){\circle*{0.3}}
\put(2,-7.5){\circle*{0.3}}
\put(2,-6.5){\line(0,-1){3.5}}
\put(2.5,-7){\circle*{0.1}}
\put(2.75,-7){\circle*{0.1}}
\put(3,-7){\circle*{0.1}}
\put(2.5,-7.5){\circle*{0.1}}
\put(3,-7.5){\circle*{0.1}}
\put(2.75,-7.5){\circle*{0.1}}
\put(3.5,-7){\circle*{0.3}}
\put(3.5,-7.5){\circle*{0.3}}
\put(1,-9){\circle*{0.3}}
\put(4,-7){\circle*{0.3}}
\put(0.9,-8.5){\line(1,0){0.2}}
\put(3.5,-6.5){\line(0,-1){3.5}}
\put(4,-6.5){\line(0,-1){3.5}}
\put(1,-6.5){\line(1,0){3}}

\put(0.2, -8){\vector(-1, 0){1.5}}

\put(-4,-6){{\footnotesize $E$}}
\put(-3.95,-10.5){{\footnotesize $\hat{\mu}$}}
\put(-5,-7){\circle*{0.3}}
\put(-2,-7.5){\circle*{0.3}}
\put(-5,-6.5){\line(0,-1){3.5}}
\put(-4.5,-7){\circle*{0.3}}
\put(-4.6,-7.5){\line(1,0){0.2}}
\put(-4.5,-8){\circle*{0.3}}
\put(-4.5,-6.5){\line(0,-1){3.5}}
\put(-4,-7){\circle*{0.3}}
\put(-4,-7.5){\circle*{0.3}}
\put(-4,-6.5){\line(0,-1){3.5}}
\put(-3.5,-7){\circle*{0.1}}
\put(-3.25,-7){\circle*{0.1}}
\put(-3,-7){\circle*{0.1}}
\put(-3.5,-7.5){\circle*{0.1}}
\put(-3,-7.5){\circle*{0.1}}
\put(-3.25,-7.5){\circle*{0.1}}
\put(-2.5,-7){\circle*{0.3}}
\put(-2.5,-7.5){\circle*{0.3}}
\put(-2,-8.5){\circle*{0.3}}
\put(-2,-7){\circle*{0.3}}
\put(-2.1,-8){\line(1,0){0.2}}
\put(-2.5,-6.5){\line(0,-1){3.5}}
\put(-2,-6.5){\line(0,-1){3.5}}
\put(-5,-6.5){\line(1,0){3}}
\end{picture}

\begin{table}
\begin{center}
\begin{tabular}{|p{3cm}|p{3.8cm}|p{5cm}|} \hline
$\mu$ & Conditions & $\mu^{\diamond}$ \\ \hline
$\langle i_{3,2}\rangle$ & $i=e-1,e\ge3$ & $\langle 1_{2^{2}},2 \rangle$ \\
 & $1\leq i \leq e-2,e\ge3$ & $\langle 0, (e-i)_{2^{2}} \rangle $ \\ 
 &$i=1,e=2$ &$\langle 1_{3,2}\rangle$\\ \hline
$\langle j,i_{2^{2}} \rangle$ & $j=0,i\ge2,e\ge3$ & $\langle (e-i)_{3,2}\rangle $ \\
& $1\leq j \leq i-2,e\ge4$ & $\langle (e-i)_{2^{2}}, e-j \rangle$ \\ \hline
$\langle i_{2^{2}},j \rangle$& $j\geq i+2\geq3,e\ge4$ & $\langle e-j, (e-i)_{2^{2}} \rangle$ \\
& $j=i+1, i=1,e\ge3$ & $\langle (e-1)_{3,2} \rangle $ \\
& $j=i+1, i\geq 2,e\ge4$ & $\langle (e-i)_{2^{2}}, e-i+1 \rangle$ \\ \hline
$\langle i_{5} \rangle$ &$i=1,e=2$ &$\langle 1_{5} \rangle$ \\
&$i=2,e=3$ &$\langle 1_{3},2_{2} \rangle$ \\
&$i=1,e=3$ &$\langle 0_{3},2_{2} \rangle$ \\
&$i=3,e=4$ & $\langle 1_{2},2_{2},3 \rangle$\\
&$i=2,e=4$ &$\langle 0_{2},2_{2},3 \rangle$ \\
&$i=1,e=4$ &$\langle 0_{2},1_{2},3 \rangle$ \\
& $i=4,e=5$ & $\langle 1_{2},2,3,4 \rangle$ \\
& $i=3,e=5$ & $\langle 0_{2},2,3,4 \rangle$ \\
 & $i=2,e=5$ & $\langle 0_{2},1,3,4 \rangle$ \\
 & $i=1,e=5$ & $\langle 0_{2},1,2,4 \rangle$ \\
& $i=e-1,e\geq6$ & $\langle 1,2,3,4,5 \rangle$\\
& $i=e-2,e\geq6$ & $\langle 0,2,3,4,5 \rangle$ \\
& $i=e-3,e\geq6$ & $\langle 0,1,3,4,5 \rangle$ \\
& $i=e-4,e\geq6$ & $\langle 0,1,2,4,5 \rangle$ \\
& $1\leq i\leq e-5,e\geq6$ & $\langle 0,1,2,3,e-i \rangle$ \\ \hline
$\langle i_{4}, i+1 \rangle$ &$i=1,e=3$ &$\langle 0_{2},2_{2,1} \rangle$ \\
&$i=2,e=4$ &$\langle 0_{2},2_{1^{2}},3 \rangle$ \\
& $i=e-2,e\ge5$ & $\langle 0,2_{1^{2}},3,4\rangle$ \\
& $2\le i\le e-3,e\ge5$ & $\langle 0,1,(e-i)_{1^{2}},e-i+1 \rangle$ \\
& $i=1,e\ge4$ & $\langle 0_{2},1,(e-1)_{1^{2}}\rangle$\\ \hline
$\langle 0, i_{4} \rangle$ & $i=1,e=2$ & $\langle 0,1_{4} \rangle$\\
&$i=2,e=3$ &$\langle 0,1_{2},2_{2} \rangle$ \\
&$i=1,e=3$& $\langle 0_{2},1,2_{2} \rangle$\\
& $i=e-1,e\ge4$ & $\langle 0,1_{2},2,3 \rangle$ \\
& $i=e-2,e\ge4$ & $\langle 0,1,2_{2},3\rangle$ \\
& $ i= e-3,e\ge5$ & $\langle 0,1,2,3_{2}\rangle$ \\
& $2\le i \le e-4,e\ge 6$ & $\langle0,1,2,(e-i)_{2}\rangle$\\
& $i=1,e\ge4$ & $\langle 0,1,2,(e-1)_{2}\rangle$ \\ \hline
$\langle i_{3}, (i+1)_{1^{2}} \rangle$ & $1\le i\le e-2,e\ge3$ & $\langle 0,(e-i-1)_{1^{2}},(e-i)_{2} \rangle$ \\ \hline

\end{tabular}
\end{center}
\end{table}

\begin{table}
\begin{center}
\begin{tabular}{|p{3cm}|p{3.8cm}|p{5cm}|} \hline
$\mu$ & Conditions & $\mu^{\diamond}$ \\ \hline
$\langle i_{3}, (i+1)_{2}\rangle$ & $i=e-2, e\ge7$ & $\langle 2,3,4,5,6 \rangle$ \\
& $i=e-3, e\ge7$ & $\langle 0,3,4,5,6 \rangle$ \\
& $i=e-4, e\ge7$ & $\langle 0,1,4,5,6 \rangle$ \\
& $i\le e-5, e\ge7$ & $\langle 0,1,2,e-i,e-i+1 \rangle$ \\ 
& $i=4, e=6$&$\langle 2_{2},3,4,5\rangle$\\
& $i=3, e=6$&$\langle0_{2},3,4,5\rangle$\\
&$i=2,e=6$&$\langle0_{2},1,4,5\rangle$\\
&$i=1,e=6$&$\langle0_{2},1,2,5\rangle$\\
& $i=3,e=5$&$\langle2_{2},3_{2},4\rangle$\\
&$i=2,e=5$&$\langle0_{2},3_{2},4\rangle$\\
&$i=1,e=5$&$\langle0_{2},1_{2},4\rangle$\\
&$i=2,e=4$&$\langle 2_{3},3_{2} \rangle$\\
&$i=1,e=4$&$\langle 0_{3},3_{2} \rangle$\\
&$i=1,e=3$&$\langle 2_{5} \rangle$\\ \hline
$\langle 0_{2}, (e-1)_{3} \rangle$ & $i=e-1,e\ge3$ & $\langle 0,1_{2,1},2 \rangle$ \\ \hline
$\langle 0_{3}, (e-1)_{2} \rangle$ & $i=e-1, e\ge6$ & $\langle 1_{2},2,3,4 \rangle$ \\ 
&$i=4,e=5$&$\langle1_{2},2_{2},3\rangle$\\ 
&$i=3,e=4$&$\langle1_{3},2_{2}\rangle$\\ 
&$i=2,e=3$&$\langle1_{5}\rangle$\\ \hline
$\langle i_{3}, i+1,i+2\rangle$ & $i=e-3,e\ge6$ & $\langle0,3_{1^{2}},4,5\rangle$ \\
& $3\le i\le e-4,e\ge7$ & $\langle 0,(e-i)_{1^{2}},e-i+1,e-i+2\rangle$ \\ 
& $i=2,e\ge5$ & $\langle 0_{2},(e-2)_{1^{2}},e-1\rangle$\\
&$i=1,e\ge4$ & $\langle 0_{2},(e-1)_{1,2}\rangle$\\ \hline
$ \langle 0,i_{3},i+1\rangle$ & $i=1,e\ge3$ & $\langle 0,1, (e-1)_{3}\rangle$ \\ 
& $2\le i\le e-2,e\ge4$ & $\langle 0, (e-i)_{2,1},e-i+1 \rangle$ \\ \hline
$\langle0,1,i_{3}\rangle$&$i=e-1,e\ge3$&$\langle0,1_{3},2\rangle$\\ 
&$2\le i\le e-2,e\ge4$&$\langle 0,1,(e-i)_{3}\rangle$\\ \hline
$\langle i_{2},(i+1)_{2},i+2\rangle$ & $i=4,e=7$ & $\langle 3_{2},4,5,6\rangle$ \\
&$i=e-3,e\ge 8$&$\langle 3,4,5,6,7\rangle$\\ 
&$i=3,e=7$&$\langle0_{2},4,5,6\rangle$\\
&$i=e-4,e\ge8$&$\langle0,4,5,6,7\rangle$\\
&$3\le i\le e-5,e\ge8$&$\langle0,1,e-i,e-i+1,e-i+2\rangle$\\
&$i=2,e\ge 7$&$\langle0_{2},1,e-2,e-1\rangle$\\ \hline
\end{tabular}
\end{center}
\end{table}

\begin{table}
\begin{tabular}{|p{3cm}|p{3.8cm}|p{5cm}|} \hline
$\mu$ & Conditions & $\mu^{\diamond}$ \\ \hline
$\langle i_{2},(i+1)_{2},i+2\rangle$&$i=2,e=6$&$\langle0_{2},4_{2},5\rangle$\\
&$i=3,e=6$&$\langle3_{2},4_{2},5\rangle$, self-dual\\
&$i=2,e=5$&$\langle3_{3},4_{2}\rangle$\\
&$i=1,e=4$&$\langle3_{5}\rangle$\\ 
&$i=1,e=5$&$\langle0_{3},4_{2}\rangle$\\
&$i=1,e\ge6$&$\langle0_{2},1_{2},e-1\rangle$\\ \hline
$\langle0,i_{2},(i+1)_{2}\rangle$ &$i=e-2,e\ge5$& $\langle0,2_{2},3,4\rangle$\\
&$2\le i\le e-3,e\ge5$&$\langle0,1,(e-i)_{2},e-i+1\rangle$\\
&$i=1,e\ge4$&$\langle0_{2},1,(e-1)_{2}\rangle$\\ 
&$i=2,e=4$&$\langle0,2_{2},3_{2}\rangle$\\ 
&$i=1,e=3$&$\langle0,2_{4}\rangle$\\ \hline
\end{tabular}
\end{table}
\textit{Acknowledgements}. This paper was written under the supervision of Kai Meng Tan at the National University of Singapore. The author would like to thank Prof Tan for his many helpful comments and guidance. The author is also grateful for the financial support given by the Agency for Science, Technology and Research.


\begin{thebibliography}{99}
\bibitem{ariki}
S. Ariki,
On the decomposition numbers of the Hecke algebra of G(m,1,n),
\emph{J. Math. Kyoto Univ.} \textbf{36} 
(1996), 789–808.

\bibitem{branching}
J. Brundan, A. Kleshchev,
Representation theory of the symmetric groups and their double covers,
in: \emph{Groups,
Combinatorics \& Geometry Durham 2001} (pp. 31–53), World Sci. Publishing, River Edge, NJ, 2003.

\bibitem{positive}
J. Chuang, H. Miyachi, K.M. Tan,
Kleshchev's decomposition numbers and branching coefficients in the fock space,
\emph{Trans. Amer. Math. Soc.} \textbf{360} (2008), 1179-1191.

\bibitem{wt 3}
M. Fayers,
Decomposition numbers for weight three blocks of symmetric groups and Iwahori–Hecke algebras,
\emph{Trans. Amer. Math. Soc.} \textbf{360(3)} (2008), 1341-1376.

\bibitem{extended James}
M. Fayers,
An extension of James's Conjecture,
\emph{Int. Math. Res. Notices} (2007), no. 10 Art. ID rnm032.

\bibitem{wt 4}
M. Fayers,
James’s Conjecture holds for weight four blocks of Iwahori–Hecke algebras,
\emph{Journal of Algebra} \textbf{317} (2007), 593–633.

\bibitem{glnq}
G. James,
The decomposition matrices of $GL_{n}(q)$ for $n\le10$,
\emph{Proc. London Math. Soc.} \textbf{60(3)} (1990), 225–265.

\bibitem{Jantzen Schaper}
G. James and A. Mathas,
A \textit{q}-analogue of the Jantzen-Schaper theorem,
\emph{Proc. London Math. Soc.} \textbf{74(3)} (1997), 241-274.

\bibitem{llt}
A. Lascoux, B. Leclerc and J.-Y. Thibon,
Hecke algebras at roots of unity and crystal bases of quantum affine algebras,
\emph{Comm. Math. Phys.} \textbf{181}
(1996), 205–263.

\bibitem{LLT}
B. Leclerc,
Symmetric functions and the Fock space,
\emph{Symmetric Functions 2001: Surveys of Developments and Perspectives}.

\bibitem{mathas book}
A. Mathas,
Iwahori–Hecke Algebras and Schur Algebras of the Symmetric Group,
\emph{University Lecture Series} \textbf{15},
American Mathematical Society, 1999.


\bibitem{Mullineux}
G. Mullineux,
Bijections on p-regular partitions and p-modular irreducibles of the symmetric groups,
\emph{J. London Math. Soc.} \textbf{20(2)} (1979), 60–66.

\bibitem{Richards}
M. Richards,
Some decomposition numbers for Hecke algebras of general linear groups,
\emph{Math. Proc. Cambridge Philos. Soc.} \textbf{119} (1996), 383-402.

\bibitem{Hansen}
S. Ryom-Hansen,
The Schaper formula and the Lascoux, Leclerc and Thibon-algorithm,
\emph{Letters in Mathematical Physics} \textbf{64}
(2003), 213-219.

\bibitem{beyond}
K.M. Tan,
Beyond Rouquier partitions,
\emph{Journal of Algebra} \textbf{321}
(2009), 248-263.

\bibitem{Williamson}
G. Williamson,
Schubert calculus and torsion explosion,
\emph{J. Amer. Math. Soc.} \textbf{30} (2017).
\end{thebibliography}
\end{document}